\documentclass[11pt, reqno]{amsart}

\usepackage{latexsym}
\usepackage{amsfonts}
\usepackage{amsmath}
\usepackage{url}
\usepackage{graphicx}
\usepackage{color}
\usepackage{bbm}
\definecolor{myblue}{rgb}{0.09,0.32,0.44} 
\usepackage[pdfborder={0 0 0},pdfborderstyle={},colorlinks=true,linkcolor=myblue,citecolor=myblue,urlcolor=blue]{hyperref}

\usepackage{epsfig,psfrag,verbatim,amsfonts}

\usepackage{amssymb,amsmath}

\def\bl{\begin{lemma}}
\def\el{\end{lemma}}
\def\bth{\begin{theorem}}
\def\eth{\end{theorem}}
\def\bc{\begin{corollary}}
\def\ec{\end{corollary}}
\def\bcj{\begin{conjecture}}
\def\ecj{\end{conjecture}}
\def\bpr{\begin{proposition}}
\def\epr{\end{proposition}}
\def\bde{\begin{definition}}
\def\ede{\end{definition}}

\newcommand{\diam}{\mbox{\rm diam}}
\newcommand{\girth}{\mbox{\rm girth}}

\newcommand{\SL}{{\mathrm{SL}}}
\newcommand{\Cay}{{\mathrm{Cay}}}

\newtheorem{theorem}{Theorem}[section]
\newtheorem{definition}{Definition}[section]
\newtheorem{lemma}[theorem]{Lemma}

\newtheorem{corollary}[theorem]{Corollary}
\newtheorem{proposition}[theorem]{Proposition}
\newtheorem{conjecture}[theorem]{Conjecture}

\newtheorem*{theorem*}{Theorem}

\theoremstyle{definition}
\theoremstyle{remark}
\newtheorem{remark}[theorem]{Remark}
\numberwithin{equation}{section}

\begin{document}
\title{Expander spanning subgraphs with large girth}
\author{Itai Benjamini Mikolaj Fraczyk and G\'abor Kun}

\email{itai.benjamini@gmail.com}
\email{mfraczyk@math.uchicago.edu}
\email{kungabor@renyi.hu}

\date{03.12.2021}

\address{The Weizmann Institute, Rehovot, Israel}

\address{Department of Mathematics, University of Chicago, 5734 S. University Avenue, 
Chicago, IL, 60637, USA}

\address{Alfr\'ed R\'enyi Institute of Mathematics, H-1053 Budapest, Re\'altanoda u. 13-15., Hungary\\ Institute of Mathematics, E\"otv\"os L\'or\'and University, P\'azm\'any P\'eter s\'et\'any 1/c, H-1117 Budapest, Hungary}

\thanks{The first author thanks the Israeli Science Foundation for support. The third author is supported by the European Research Council (ERC) Advanced Grant No. 741420, by the J\'anos Bolyai Scholarship of the Hungarian Academy of Sciences and by the \'UNKP-20-5 New National Excellence Program of the Ministry of Innovation and Technology from the source of the National Research, Development and Innovation Fund.}

\begin{abstract}

 We conjecture that in any finite graph with large Cheeger constant we can delete a proportion of edges in such a way that the remaining graph has large girth and retains good expansion properties. We prove this when the expansion is large enough in terms of the maximum degree. The condition on expansion covers, for example, large random $d$-regular graphs. Our proof relies on the Lov\'asz Local Lemma.
\end{abstract}

\maketitle

\section{Introduction}
 Families of expander graphs play a central role in modern graph theory and geometric group theory. We take a closer look at the interaction of the expansion condition and the large girth condition. At a first glance these two properties point in the same direction. A graph with many short loops will tend to have a small spectral gap and, as a consequence, small expansion. This intuition is confirmed in the case of Ramanujan graphs which are $d$-regular graphs with spectral gap at least $d-2\sqrt{d-1}$, nearly best possible. It was proved in \cite{AGV} that such graphs have a small number of short cycles. In contrast, many known constructions of expander graphs produce sequences of graphs of bounded girth. For example, the two arguably most popular constructions - Cayley graphs of quotients of groups with  Kazhdan's Property $(T)$ and random $d$-regular graphs - exhibit this behaviour. It is even possible to construct sequences of Ramanujan graphs of girth $1$ \cite{Gl}. In some contexts (see for example \cite{ALM}) it is desirable to have a sequence of expander graphs with girth tending to infinity, which excludes the aforementioned constructions. Our work tries to remedy this problem. In our main results, Theorems \ref{main}, \ref{cor-Spectral}, we prove that for any expander sequence satisfying a mild condition on the expansion or on the number of short cycles it is always possible to "thin out" the graphs by deleting a proportion of edges in such a way that the girth is large and the Cheeger constant stays above certain explicit positive threshold. 

\subsection{Main results.} Let $G$ be a finite graph with vertex set $V(G)$. The maximum degree of $G$ is denoted by $\Delta(G)$. The Cheeger constant of $G$ is defined as \[h(G):=\inf_{\substack{S\subset V(G),\\ 0<|S|\leq |V(G)|/2}} \frac{|\partial_G S|}{|S|},\] where $\partial_G S$ denotes the edge boundary of $S$. A spanning subgraph of $G$ is a subgraph containing all the vertices of $G$. We conjecture the following:

\begin{conjecture}\label{mainConj} For every $\varepsilon>0$ and natural number $\Delta$ there exist $\kappa, K>0$ such that if $G$ is a finite graph with $\Delta(G) \leq \Delta$ and
$h(G)>\varepsilon$ then $G$ admits a spanning subgraph $H$ with \[h(H) > \kappa\quad \textrm{ and }\quad\diam(H) <K \girth(H).\]
\end{conjecture}

We provide some motivation and examples in section \ref{sec-motivatiom}.
 We can partially prove the conjecture when the expansion is somewhat large as a function of the maximum degree, see Theorem~\ref{cor-Spectral}. First we prove a theorem about graphs that contain few short cycles locally. By a cycle we mean a $2$-regular simple finite subgraph without a starting vertex or an orientation.

\begin{theorem}~\label{main}
 Let $\Delta, g$ be positive integers, $\delta$ positive and $G$ a finite graph. Assume that $\Delta(G) \leq \Delta$ and for every $k \leq g$ every vertex $x$ is contained 
in at most $\big( \frac{\Delta}{8\delta} \big)^k$ cycles of length $k$. Then $G$ admits a spanning subgraph $H$ with girth at least $g$ such that 
\[|\partial_H S| \geq \frac{\delta}{2\Delta}|\partial_G S|-(2\log(\Delta)+4)|S|\] holds  for every $S \subseteq V(G)$.
\end{theorem}
All the logarithms are the natural logarithms.
Note that the theorem is meaningless if $\delta<4\log(\Delta)$ because then $H$ can be the edgeless spanning subgraph.


\begin{corollary}~\label{Cheeger}
Let $\Delta, g$ be positive integers, let $\delta>0$ and let $G$ be a finite graph. Assume that $\Delta(G) \leq \Delta$, $h(G) \geq \frac{8\Delta(\log(\Delta)+2)}{\delta}$ and for every $k \leq g$ every vertex $x$ is contained in at most $\big( \frac{\Delta}{8\delta} \big)^k$ cycles of length $k$. Then $G$ admits a spanning subgraph $H$ with girth at least $g$ and $h(H) \geq \frac{\delta}{4\Delta} h(G)$.
\end{corollary}


The number of cycles can be controlled using the spectral conditions so it is possible to combine the condition on the cycles and the Cheeger constant into one. Let $\mathcal A, \mathcal D, \mathcal L$ be the adjacency operator on $G$, the pointwise multiplication by the degree and the combinatorial Laplacian on $G$, respectively. These are  bounded linear operators on $ L^2(V(G))$ given by 
\begin{equation}\label{ADLdef} \mathcal A f(v)=\sum_{v'\sim v} f(v'),\quad \mathcal D f(v)=\deg(v)f(v),\quad \mathcal L f(v)=\sum_{v'\sim v} (f(v)-f(v')).\end{equation}
We note that $\mathcal L=\mathcal D-\mathcal A$. For a finite graph $G$ with $n$ vertices we shall list the eigenvalues of $\mathcal L$ by $0=\lambda_1\leq \lambda_2\leq \ldots \leq \lambda_n$. 
In a finite graph $G$, the Cheeger constant $h(G)$ and the second eigenvalue $\lambda_2(G)$ of $\mathcal L$ are related by the following inequality (see, e.g., \cite{Dod})
\[\lambda_2(G)/2\leq h(G)\leq \sqrt{2\Delta\lambda_2(G)}.\]

 We state our main theorem about graphs with somewhat large expansion.

\begin{theorem}\label{cor-Spectral}
Let $\Delta, g$ be positive integers and let $G$ be a finite graph. Assume that $\Delta(G) \leq \Delta$, $\lambda_2(G) \geq \Delta-\frac{\Delta}{16(\log \Delta+2)}$ and $|V(G)|$ is at least $\max\{ \Delta^{k+1}(\Delta-\lambda_2)^{-k-1} | k=0,\ldots, \max\{g, 8\log \Delta\}\}$. Then $G$ admits a spanning subgraph $H$ of girth at least $g$  with \[h(H) \geq \left(\frac{\lambda_2(G)(\Delta-\lambda_2(G))^{-1}}{64}-2\log \Delta -4\right).\]
In particular, if $\lambda_2(G) \geq \Delta-\frac{\Delta}{192(\log \Delta+2)}$, then $h(H)\geq 0.99(\log \Delta +2).$
\end{theorem}
\begin{remark} Ramanujan graphs of degree $d$  have $\lambda_2(G)\geq d-2\sqrt{d-1}$ and the right hand side is larger than $d-\frac{d}{16(\log d+2)}$ for large enough $d$, so Theorem \ref{cor-Spectral} is not vacuous.  It even gives a subgraph with expansion of order $\Omega(\sqrt d)$, on par with Ramanujan graphs. Random $d$-regular graphs are almost Ramanujan in the sense that $\lambda_2(G)= d-{2\sqrt{d-1}}+O(1)$ with high probability \cite{Puder}. Therefore, by Theorem \ref{cor-Spectral}, large enough random $d$-regular graphs admit spanning subgraphs of girth proportional to the diameter and expansion of order  \[\left(\frac{d-2\sqrt{d-1}+O(1)}{64(2\sqrt{d-1}+O(1))}-2\log d -4\right)= \Omega(\sqrt{d}).\]
\end{remark}

\subsection{Idea of the proof.} Our proof is an advanced version of the proof of Theorem 7 in \cite{K}, where the goal was to find a spanning subgraph $H$ with large girth, but there was no condition on the Cheeger constant of $H$. We also need to check that there are no sets with small expansion, and our argument here is similar to the approach of Bilu and Linial from \cite{BL}. Both large girth and expansion are verified using the Lov\'asz Local Lemma. Given that the number of conditions one has to check is a priori exponential in the number of vertices, we find it interesting that the Local Lemma helps to have a large Cheeger constant.

If the expansion is small a natural approach to attack the conjecture is to work with a power graph. Let $A$ be the adjacency matrix of $G$. The power graph $G^{(p)}$ is the graph on the same vertex set with the adjacency matrix $A^p$. When $G$ is $d$-regular and $p$ is odd, we have $(d^p-\lambda_2(G^{(p)}))=(d-\lambda_2(G))^p$. By taking large enough power we can always fall into the range covered by Theorem \ref{cor-Spectral}. This can be used to get some information on the original graph. A $K$-Lipschitz subgraph of a graph $G$ is a graph $H$ on the same vertex set $V(G)$ such that any edge in $H$ connects points that are at distance at most $K$ in $G$. Note that we do not require $H$ to be an actual subgraph of $G$. For example, any subgraph of the $p$-th power graph $G^{(p)}$ is a $p$-Lipshitz subgraph of $G$.
A spanning Lipschitz subgraph with large girth $H$ was found in \cite{K}, provided that $G$ is 
regular.
The core of the problem is that we do not know that if a power graph, say $G^{(2)}$, admits a spanning expander subgraph with large girth then so does $G$.

\subsection{Motivating examples}\label{sec-motivatiom}
Benjamini and Schramm \cite{BS} showed that infinite graphs with positive Cheeger constant contain a tree with positive Cheeger constant, solving a problem of Deuber, Simonovits and T. S\'os  \cite{DSS}. 

In \cite{K} a factor of iid version of this theorem was proved for Cayley graphs of non-amenable groups with large Cheeger constant. 
The measurable solution of Gaboriau and Lyons \cite{GL} to the von Neumann problem requires less: there is no bound on the length
of the edges in the spanning forest (but it should be a factor of iid) and it will be the Schreier graph of an almost free action of the free group. 
On the other hand, Gaboriau and Lyons showed that this action is ergodic. Their result had many applications to operator algebras, 
see Houdayer's Bourbaki seminar paper \cite{H}.

Many interesting examples of expander families are provided by Cayley graphs of congruence quotients of arithmetic groups. For example, for any  $d\geq 2$ and a set $S \subset {\rm SL}_d(\mathbb Z)$ spanning a Zariski dense subgroup the family of Cayley graphs  ${\rm Cay}(\SL_d(\mathbb{F}_p), S)$ is an expander family with girth $\Omega (\log p)$. This is a rather deep fact originating in the work of Bourgain and Gamburd \cite{BoGa}, expanded upon by Bourgain and Varj\'u \cite{BV}. For some of these arithmetic examples we can find evidence toward Conjecture \ref{mainConj} by hand. Let us take a closer look at $\SL_2(\mathbb Z).$ Let $S\subset \SL_2(\mathbb Z)$ be a finite symmetric set generating a Zariski dense subgroup. By the work of Bourgain and Gamburd \cite{BoGa} it is known that the sequence of graphs $G_p:=\Cay(\SL_2(\mathbb Z/p\mathbb Z), S)$ is an expander sequence as $p$ varies among sufficiently big primes. Let $\Gamma$ be the subgroup generated by $S$. The conjecture predicts that we should be able to find spanning subgraphs of $G_p$ of large girth which are still good expanders. In this example we can do something slightly weaker. We can find a sequence of expander graphs as spanning Lipschitz subgraphs $H_p$ of girth proportional to $\log p$. To find such graphs take a finite power of $S$ that contains two elements, say $a$ and $b$, that generate a free Zariski dense subgroup. This is always possible by the Tits alternative. Put $H_p:=\Cay(\SL_2(\mathbb Z/p\mathbb Z), \{a,b,a^{-1},b^{-1}\}).$ By \cite{BoGa} $H_p$ is an expander sequence. Since $\langle a,b\rangle$ is free, the girth must go to infinity. Indeed, any relation on $a,b$ that holds modulo infinitely many primes would have to hold in $\SL_2(\mathbb Z)$, which contradicts the freeness. In fact, it is easy to show that the girth of these graphs grows as $\Omega(\log p)$. The expansion constant provided by the method of Bourgain-Gamburd \cite{BoGa} and the following improvements (see e.g. \cite{BV}) is not explicit and likely very small. In comparison, our Theorem \ref{cor-Spectral} allows one to find expanding subgraphs of $\Cay(\SL_2(\mathbb Z/p\mathbb Z),\Sigma)$ of girth $\gg \log p$ of expansion constant $\gg \log d$ provided that the initial generating set $\Sigma$ has sufficiently large expansion, what can always be arranged by taking a power of $\Sigma$.

Arzhantseva and Biswas \cite{AB} constructed an expander sequence of Cayley graphs of $\SL_d(\mathbb{F}_p)$ with bounded diameter-girth ratio. We give an alternative of their theorem finding a spanning subgraph of any good expander graph with bounded diameter-girth ratio. 

Sequences of expanders with increasing girth are interesting from the point of view of quantum ergodicity. Let $H_n$ be an expander sequence of growing size and girth. In \cite[Thm 1]{ALM} Anantharaman and Le Masson proved that in such a sequence the mass of a typical Laplace eigenvalue is well equidistributed over the vertices of the graph. Given any expander sequence $G_n$, the spanning subgraphs $H_n$ formed using Corollary~\ref{Cheeger} and Theorem~\ref{cor-Spectral} always fall into the scope of Anantharaman and Le Masson's result. For quantitative statements and interesting examples we refer to the paper of Alon, Ganguly and Srivastava \cite{ASS}.

Thomassen made the following conjecture \cite{T}.
Does for every $d$ and $g$ there exist a $D$ such that every finite graph with average degree 
at least $D$ contains a subgraph with average degree at least $d$ and girth at least $g$?
Our conjecture would imply Thomassen's conjecture for graphs with large enough Cheeger constant.

\section{The Lov\'asz Local Lemma}

One of the most useful basic facts in probability is the following. If there is a finite set of mutually independent events that each of 
them holds with positive probability then the probability that all events hold simultaneously is still positive, although small. 
The Lov\'asz Local Lemma allows one to show that this statement still holds in case of rare dependencies.  

We will use the so-called {\bf variable version} of the lemma: We will consider a set of mutually independent 
random variables. Given an event $A$ determined by these variables 
we will denote by $vbl(A)$ the unique minimal set of variables that determines the event $A$: 
such a set clearly exists. Note that given the events $A, B_1, \dots ,B_m$ if $vbl(A) \cap vbl(B_i) = \emptyset$ 
for every $1 \leq i \leq m$ then $A$ is mutually independent of all the events $B_1, \dots ,B_m$.

\begin{lemma}\cite{EL}
Let $\mathcal{V}$ be a finite set of mutually independent random variables in a probability space.
Let $\mathcal{A}$ be a finite set of events determined by these variables. If there exists an assignment
$x: \mathcal{A} \rightarrow (0,1)$ such that for every $A \in \mathcal{A}$ we have
$\mathbb{P}(A) \leq x(A) \prod_{vbl(A) \cap vbl(B) \neq \emptyset} (1-x(B))$ then $\mathbb{P} \big( \bigwedge_{A \in \mathcal{A}} \overline{A} \big) \geq  \prod_{A \in \mathcal{A}} (1-x(A)).$
\end{lemma}

\section{The proof of Theorem~\ref{main}}

A large girth spanning subgraph with large minimum degree was obtained under the condition on the number of cycles in \cite{K}, see Theorem 7. 
Now we have an expansion condition instead of the bound on the minimum degree.
With the application of Lov\'asz Local Lemma in mind it would be natural to consider the following probability distribution on spanning subgraphs: every edge is chosen with the same probability, independently. This is more or less what we are going to do, but for technical reasons it will be more convenient to work with a random directed subgraph first and turn it into an undirected one at the end of the argument.

Let the parameters $\Delta, \delta, g$ be as in the statement of Theorem~\ref{main}. We consider a random directed graph $D$ on the vertex set $V(G)$ constructed as follows.
For every $(x,y) \in E(G)$ we add $(x,y)$ to $E(D)$ with probability $\frac{\delta}{\Delta}$, and we add $(y,x)$ to $E(D)$ with probability $\frac{\delta}{\Delta}$.
All these choices are independent. The undirected graph $H$ on the vertex set $V(G)$ is determined by $D$ according to the following rule: an undirected edge is present in $E(H)$ if at least one of the corresponding directed edges is in $E(D)$, that is, \\
$E(H)=\{(x,y): (x,y) \in E(D) \text{ or } (y,x) \in E(D)\}$.

We will apply the Local Lemma to this probability space. The set of variables $\mathcal{V}$ will correspond to the directed edges of $G$: 
Two directed edges correspond to every edge ($G$ has no loops).
We will call a cycle in $G$ {\it short} if it is shorter than $g$. The ''bad events'' of $\mathcal{A}$ correspond to either short cycles
or connected subsets of vertices. Let $\mathcal{C}$ denote the set of short cycles in $G$, and $\mathcal{S}$ the set of connected
subsets of vertices. For $C \in \mathcal{C}$ let $A_C$ denote the event that $E(C) \subseteq E(H)$.
Given a set $S \in \mathcal{S}$ let $A_S$ denote the event that \[|\partial^{\rm out}_D S| < \frac{\delta}{2\Delta} |\partial_G (S)|-(2\log(\Delta)+4)|S|,\] where
$\partial^{\rm out}_D S$ denotes the set of directed edges in $D$ with starting point in $S$ and endpoint not in $S$. Note that if this inequality holds for every connected set $S$ then it holds for every set.
To prove the theorem it suffices to show that with positive probability the set of events $\{ A_C: C \in \mathcal{C} \} \cup \{ A_S: S \in \mathcal{S}\}$ can be avoided. 

As advertised before we are going to use the Local Lemma. We will choose the values $x(A_C),x(A_S)$ later. The first order of business is to bound the probabilities of events $A_C,A_S$.

{\bf Claim 1:} Let $C \in \mathcal{C}$ be a cycle of length $k$ in $G$. Then $\mathbb{P} (A_C) \leq (\frac{2\delta}{\Delta})^k$.

{\bf Claim 2:} For every $S \in \mathcal{S}$ the inequality $\mathbb{P}(A_S) \leq (16\Delta)^{-|S|}$ holds.
Since the choices or edges are independent, the first Claim requires no explanation.
\textbf{ Proof of Claim 2:}
Consider the edges in $\partial_G (S)$. For every such edge we make a choice if we include it in $\partial^{out}_D S$ (with the right direction).
These are independent choices. We apply the Chernoff bound. 
Given a real number $\lambda$ consider the expected value
\begin{align*}\mathbb{E}\left( \exp \big( \lambda |\partial^{out}_D S| \big)\right)= \sum_{k=0}^{|\partial_G (S)|} \binom{|\partial_G (S)|}{k} \bigg( \frac{e^\lambda \delta}{\Delta}\bigg)^k  \bigg( \frac{\Delta-\delta}{\Delta}\bigg)^{|\partial_G (S)|-k}= \\
\bigg( 1 + \frac{(e^\lambda-1) \delta}{\Delta}\bigg)^{|\partial_G (S)|}.\end{align*}
Since $1 + \frac{(e^\lambda-1) \delta}{\Delta} \leq \exp \Big(\frac{(e^\lambda-1)\delta}{\Delta} \Big)$, the Markov inequality implies 
\[\mathbb{P}\Bigg(\exp \big(\lambda |\partial^{\rm out}_D S| \big) \geq (10\Delta)^{|S|} \exp \Big( \frac{(e^\lambda-1)\delta}{\Delta}|\partial_G (S)| \Big) \Bigg) \leq (16\Delta)^{-|S|}.\]
Choosing $\lambda=-\log(2)$ we get that
\[\mathbb{P}\Big(|\partial^{\rm out}_D S| \leq \frac{-\log(16\Delta)}{\log(2)}|S|  + \frac{\delta}{2\log(2)\Delta}|\partial_G (S)| \Big) \leq (16\Delta)^{-|S|}\]
The claim follows, since $1/2 < \log(2) < 1$.
\begin{lemma}\label{lem-setnum}
Let $G$ be a graph, $x$ a vertex of $G$ and $s$ a positive integer.
Then the number of connected subsets of vertices in $V(G)$ of size $s$ containing $x$ is at most
$\Delta(G)(\Delta(G)-1)^{s-2} {2s-2 \choose s-1} \leq (4\Delta(G))^{s-1}.$
\end{lemma}

\begin{proof}
We can assume that the connected component containing $x$ has at least $s$ elements as otherwise there is nothing to prove. Given a connected subset of vertices containing $x$ we explore it by running a Depth First Search algorithm starting at $x$.
We keep track of the log of the walk. If the number of explored vertices is less than $s$ and it is possible to move to an unexplored vertex we do it, otherwise we move backwards. The exploration ends when we discovered $s$ vertices and moved back to $x$.
In total we make $(s-1)$ forward steps, since a new vertex is explored at every forward step, and $(s-1)$ backward steps.
There are at most ${2s-2 \choose s-1}$ ways to choose the order of forward and backward steps. We have at most $(\Delta(G)-1)$ choices when making a forward step,
except in the first step, when we have at most $\Delta(G)$ possibilities. There is only one choice for every backward step. The number of subgraphs of size $s$ containing $x$ is bounded by the number of possible logs, hence it is at most $\Delta(\Delta-1)^{s-2}{2s-2 \choose s-1}.$ The lemma follows.
\end{proof}

We choose the function $x$ in order to apply the Local Lemma. For every $S \in \mathcal{S}$ set $x(A_S)=((8\Delta)^{|S|}+1)^{-1}$, and for every $C \in \mathcal{C}$
set $x(A_C)=\big(\frac{4\delta}{\Delta}\big)^{|C|}$, where $|C|$ denotes the length of the cycle $C$. We have to check the condition of the Local Lemma. 

First, we bound the product of terms of the form $(1-x(A))$. We will have separate bounds for short cycles and connected subsets.  
Let $v \in V(G)$. We use that the number of cycles of length $k$ containing $v$ is at most $\big( \frac{\Delta}{8 \delta} \big)^k$.
\begin{align*}\prod_{v \in C \in \mathcal{C}} (1-x(A_C)) =&  \prod_{k=3}^{g-1} \prod_{v \in C \in \mathcal{C}, |C|=k} 
\Big(1-\Big( \frac{4\delta}{\Delta} \Big)^k \Big)\geq \prod_{k=3}^{g-1} \Big(1-\Big( \frac{4\delta}{\Delta} \Big)^k \Big)^{\big( \frac{\Delta}{8 \delta} \big)^k} \\  \geq& \prod_{k=3}^{\infty} \exp(-2^{-k}) = e^{-1/4}.\
 \end{align*}
Similarly, using Lemma \ref{lem-setnum}, we have 
\begin{align*}\prod_{v \in S \in \mathcal{S}} (1-x(A_S)) =& \prod_{k=1}^{\infty} \prod_{\substack{v \in S \in \mathcal{S}\\ |S|=k}}(1-x(A_S))
\geq\prod_{k=1}^{\infty}  \exp \big(-(8\Delta)^{-k} (4\Delta)^{k-1} \big)\\  \geq& \exp \Big( \frac{1}{4\Delta} \sum_{k=1}^{\infty} -2^{-k}\Big) = e^{-\frac{1}{4\Delta}}\end{align*}
For every $S \in \mathcal{S}$ the inequality
\begin{align*}\mathbb{P}(A_S) \leq& (16\Delta)^{-|S|} \leq ((8\Delta)^{|S|}+1)^{-1} \exp \Big( -\frac{\Delta+1}{4\Delta} \Big)^{|S|} \\ \leq& 
x(A_S) \Pi_{v \in S} \prod_{v \in C \in \mathcal{C}} (1-x(C)) \prod_{v \in T \in \mathcal{S}} (1-x(A_T))\end{align*} holds, as needed.
Similarly, for every $C \in \mathcal{C}$
\begin{align*}\mathbb{P}(A_C) \leq& \Big( \frac{2\delta}{\Delta} \Big)^{|C|} \leq \Big(\frac{4\delta}{\Delta}\Big)^{|C|} \exp \Big( -\frac{\Delta+1}{4\Delta} \Big)^{|C|}\\
 \leq& x(A_C) \prod_{v \in C} \prod_{v \in D \in \mathcal{C}} (1-x(A_D)) \prod_{v \in S \in \mathcal{S}} (1-x(S)).\end{align*}
The assumptions of the Local Lemma are satisfied so with a positive probability none of the events in $\mathcal A$ occur. This implies that the undirected graph $H$ satisfies the conclusion of Theorem~\ref{main} with positive probability.

\section{The proof of Theorem \ref{cor-Spectral}}


\begin{lemma}\label{lem-returns}
Let $G$ be a finite graph of maximum degree $\Delta$. Let $\lambda_2(G)$ be the second eigenvalue of the Laplacian operator $\mathcal L$, defined as in  (\ref{ADLdef}). Then, for every vertex $v\in V(G)$ the number of closed walks of length $k$ starting at $v$ is at most \[\begin{cases}\Delta^k|V(G)|^{-1}+(\Delta-\lambda_2(G))^k & \text{ if } k \text{ is odd,}\\
4\Delta^{k+1}|V(G)|^{-1} +2(\Delta-\lambda_2)^{k+1} & \text{ if } k \text{ is even.}\end{cases}\]
\end{lemma}
\begin{proof}
Let $\mathcal A, \mathcal D$ be the operators defined in (\ref{ADLdef}). Write $V:=V(G), \lambda_i:=\lambda_i(G).$ Let $t_k(v)$ be the number of closed walks of length $k$ starting at $v$. 
To bound $t_k(v)$ we consider the cases of odd and even $k$ separately. 

We start with the simpler case of odd $k$:
Both $\mathcal A$ and $\mathcal (\Delta-\mathcal D)$ are non-negative operators, so we have
\[ t_k(v)=\langle \mathcal A^k \mathbf 1_v,\mathbf 1_v\rangle\leq \langle (\mathcal A+\Delta - \mathcal D)^k \mathbf 1_v,\mathbf 1_v\rangle=\langle \mathcal (\Delta -\mathcal L)^k \mathbf 1_v,\mathbf 1_v\rangle.\]  
Let $f_1,\ldots ,f_n$ be the normalized eigenvectors of $\mathcal L$, with $f_1=|V(G)|^{-1/2}\mathbf 1_{V(G)}.$ We have $\mathbf 1_v=\sum_{i=1}^n \alpha_i f_i, $ with $\alpha_1=|V(G)|^{-1/2}$ and $\sum_{i=1}^n \alpha_i^2=1$. Hence 
\[ \langle (\Delta -\mathcal L) ^k \mathbf 1_v,\mathbf 1_v\rangle= \Delta^k|V(G)|^{-1}+ \sum_{i=2}^n |\alpha_i|^2 (\Delta-\lambda_i)^k\leq \Delta^k|V(G)|^{-1}+(\Delta-\lambda_2)^k.\]

If $k$ is even, the inequality $(\Delta-\lambda_i)^k\leq (\Delta-\lambda_2)^k$ might not hold. We must modify our argument, which will result in a slightly weaker bound. Choose a neighbour $v'$ of $v$. Every closed walk of length $k$ starting at $v$ can be extended uniquely to a walk of length $k+1$ ending at $v'$. Hence,
\[t_k(v)\leq \langle \mathcal A^{k+1} \mathbf 1_v,\mathbf 1_v'\rangle\leq \langle \mathcal A^{k+1}(\mathbf 1_v+\mathbf 1_{v'}),(\mathbf 1_v+\mathbf 1_{v'})\rangle. \]
 Write $F:=(\mathbf 1_v+\mathbf 1_{v'})$ using the normalized eigenbasis: \[F=\beta_1 f_1+\beta_2 f_2+\ldots \beta_n f_n.\] Then, $\beta_1=2/|V|^{1/2}$ and $\sum_{i=1}^n \beta_i^2=\|(\mathbf 1_v+\mathbf 1_{v'})\|^2=2.$ The operators $\mathcal A,\Delta-\mathcal D$ are non-negative so
\begin{align*}\langle \mathcal A^{k+1}F,F\rangle&\leq \langle (\mathcal A+\Delta-\mathcal D)^{k+1}F,F\rangle=\langle \mathcal (\Delta-\mathcal L)^{k+1}F,F\rangle\\ &= \frac{4\Delta^{k+1}}{|V|}+\sum_{i=2}^n\beta_i^2(\Delta-\lambda_i)^{k+1}\\
&< \frac{4\Delta^{k+1}}{|V|} +2(\Delta-\lambda_2)^{k+1}.\end{align*}

\end{proof}

\begin{corollary}\label{returning-bound} Let $g\in\mathbb N$ and assume that $|V(G)|\geq 4|\Delta|^{m}(\Delta-\lambda_2)^{-m}$ 
for all $0\leq m\leq \max\{g, 8\log \Delta\}+1$ and $\lambda_2\geq \Delta-\frac{\Delta}{16(\log \Delta +2)}$. Then, for every vertex $v\in V(G)$ and $k\leq g$, the number of returning walks of length $k$ starting at $v$ is at most $2^k(\Delta-\lambda_2(G))^k.$
\end{corollary}
\begin{proof}
If $k$ is odd then 
\[t_k(v)\leq \frac{\Delta^k}{|V|}+(\Delta-\lambda_2)^k\leq 2(\Delta-\lambda_2)^k\leq 2^k(\Delta-\lambda_2)^k.\]
If $k$ is even and $\max\{g, 8\log \Delta\}\geq k\geq 0$ then 
\[t_k(v)\leq \frac{4\Delta^{k+1}}{|V|}+2(\Delta-\lambda_2)^{k+1}\leq 3(\Delta-\lambda_2)^{k+1}.\]
If $k\geq 4\log \Delta$ then $2^{k-1}\geq  \frac{3\Delta}{32}\geq \frac{3\Delta}{16(\log \Delta+2)}$ so 
\[t_k(v)\leq 3(\Delta-\lambda_2)^{k+1}\leq 2^k(\Delta-\lambda_2)^k.\]
If $k\leq 4\log \Delta$ we choose $m$ natural such that $\max\{g, 8\log \Delta\}\geq km \geq 4\log \Delta$. The previous case yields 
\[t_{mk}(v)\leq 2^{mk}(\Delta-\lambda_2)^{mk}.\] The number of closed walks satisfies $t_{a+b}(v)\geq t_a(v)t_b(v)$, so 
\[t_k(v)\leq t_{mk}(v)^{1/m}\leq 2^k(\Delta-\lambda_2)^k.\]
\end{proof}

\begin{proof}[Proof of Theorem \ref{cor-Spectral}]
The number of cycles of length $k$ containing $v$ is denoted by $c_k(v)$. This is bounded by the number of closed walks of length $k$ starting at $v$. Corollary \ref{returning-bound} implies $c_k(v) \leq t_k(v)\leq 2^k (\Delta-\lambda_2(G))^k$ for each $k\leq g$. Choose $\delta=\Delta(\Delta-\lambda_2(G))^{-1}/16$. The cycle condition in Theorem \ref{main} is satisfied, so we get a spanning subgraph $H$ of girth at least $g$ such that 
\begin{align*}|\partial_H S| \geq& \frac{\delta}{2\Delta}|\partial_G S|-(2\log\Delta+4)|S|\geq \left(\frac{\delta}{4\Delta} \lambda_2(G)-(2\log \Delta+4)\right)|S|\\
=& \left(\frac{\lambda_2(G)(\Delta-\lambda_2(G))^{-1}}{64}-2\log \Delta -4\right) |S|,
\end{align*} for every subset $S\subset V(G)$ with $|S|\leq |V(G)|/2$.
If $\lambda_2(G) \geq \Delta-\frac{\Delta}{192(\log \Delta+2)}$ then $(\Delta-\lambda_2(G))^{-1}/64\geq 3(\log \Delta+2)\Delta^{-1}$. We have 
\[|\partial_H S|\geq \left(((1-\frac{1}{384})3-2)(\log\Delta+2)\right)|S|\geq 0.99(\log \Delta +2)|S|. \]

The Cheeger constant of $H$ satisfies $h(H)\geq 0.99(\log \Delta +2).$
\end{proof}

\begin{remark}
The question if our proofs can be turned into an effective algorithm is quite challenging. The natural approach is to rely on the results of Moser and Tardos
\cite{MT}. Their randomized algorithm is based on iterative resampling. They start with a random sample of the variables and until there exists a single bad event that is not avoided, they resample the variables of this event. It does not matter which bad event they choose, the expected number of resamplings is
$\sum \frac{x(A)}{1-x(A)}$, where the summation goes over all bad events. This would give a polynomial expected running time in our case.
 
Unfortunately, it is not clear, how we detect which bad event holds. While short cycles are easy to find a set with small expansion is not.
Hence, in case of  Theorem~\ref{main} we do not know any efficient approach: in case of this theorem we restrict to a family of the subsets of $V(G)$,
subsets with large expansion in $G$, and we want to decide if they all have large expansion. This might be hard in certain cases and it is not expected to have a polynomial time algorithm, not even in quantitatively weaker forms: Small Set Expansion problem, a problem known to be harder than Unique Games, can be phrased this way \cite{RS}.

On the other hand, in case of Theorem \ref{cor-Spectral} we can have an efficient algorithm that outputs an $H$ with a weaker bound 
on the Cheeger constant. If we have a set with small relative boundary then we are able to find a set with somewhat larger relative boundary. Assume that there exists a function $f\colon \mathbb R\to \mathbb R$ and a polynomial time algorithm that finds for every graph $G$ with Cheeger constant at most $\gamma$ a set with expansion at most $f(\gamma)$ (for example, one can do this using eigenvectors and the Cheeger inequalities). This makes it possible to run the Moser-Tardos algorithm. This algorithm gives a constructive proof of Corollary \ref{Cheeger} and Theorem \ref{cor-Spectral} but the bound on the Cheeger constant is $f(\gamma),$ instead of the original bound $\gamma$. 

While the Moser-Tardos approach often allows derandomization using an expected value argument choosing the values of the variables one by one, this does not seem easy in our case, since the number of bad events is exponentially large. 
\end{remark}

\begin{remark}
The Moser-Tardos algorithm is also efficient in the sense that it gives a factor of iid solution on infinite graphs, too \cite{K}.
Moreover, the choice of the edges at every vertex will depend only on a ball of constant radius centered at the vertex for all, 
but an exponentially small proportion of the vertices. It would be interesting to derandomize this algorithm. 
\end{remark}

\section{Future directions}

It would be interesting to have an efficient, deterministic algorithm that constructs the spanning subgraph $H$.

Graph sparsification turned out to be a successful method in the last decade, see the work of Spielman and Srivastava \cite{SS}. 
A Local Lemma based sparsification could be very interesting, especially if it provides a large girth 
graph in certain cases. Unfortunately, our proof does not achieve this goal, but a similar approach might work.

A measurable solution to the von Neumann problem with spectral gap and Lipschitz condition could probably have many applications, see \cite{GL, H, K} 
for details. Theorem~\ref{cor-Spectral} implies the finite analogue of this.

Are there approachable high dimensional formulations of the conjecture for the several variants of high dimensional expanders \cite{L}?
For the standard definitions and background see \cite{HLW}. 

Let us end by recalling an old conjecture: there is no sequence of finite bounded degree graphs growing  in size  to infinity, so that all the induced balls in all the graphs in the sequence are uniform expanders. 
For a related progress see \cite{FV}.

\textbf{Acknowledgment.}
The authors thank to G\'abor Pete for encouraging this collaboration, to Elad Tzalik for his suggestions to improve the paper and to Merav Parter for her comments on possible applications.

\end{document}